\documentclass[12pt,a4paper,reqno]{amsart}
\usepackage[mathcal]{eucal}
\usepackage{amssymb}
\usepackage[nice]{nicefrac}
\usepackage{hyperref}
\usepackage{color}

\theoremstyle{plain}
\newtheorem{theorem}{Theorem}[section]
\newtheorem{lemma}[theorem]{Lemma}
\newtheorem{proposition}[theorem]{Proposition}
\newtheorem{corollary}[theorem]{Corollary}

\theoremstyle{definition}

\newtheorem{problem}[theorem]{Problem}
\newtheorem{problems}[theorem]{Problems}

\newtheorem{beispiel}[theorem]{Example}
\newenvironment{example}{\begin{beispiel}%
  \pushQED{\qed}}%
  {\popQED\end{beispiel}}

\numberwithin{equation}{section}

\DeclareMathOperator{\hdim}{hdim}
\DeclareMathOperator{\hspec}{hspec}

\begin{document}

\title[Hausdorff dimensions in analytic groups]{Hausdorff dimensions in
  $p$-adic analytic groups}

\author[B. Klopsch]{Benjamin Klopsch} \address{Benjamin Klopsch:
  Mathematisches Institut, Heinrich-Heine-Universit\"at, 40225
  D\"usseldorf, Germany} \email{klopsch@math.uni-duesseldorf.de}

\author[A. Thillaisundaram]{Anitha Thillaisundaram} 
\address{Anitha Thillaisundaram: School of Mathematics and Physics,
  University of Lincoln, LN6 7TS Lincoln, England}
\email{anitha.t@cantab.net}

\author[A. Zugadi-Reizabal]{Amaia Zugadi-Reizabal} 
\address{Amaia Zugadi-Reizabal: Department of Mathematics, University
  of the Basque Country UPV/EHU, 48080 Bilbao, Spain}
\email{amaia.zr@gmail.com}


\thanks{The second author acknowledges support from the
    Alexander von Humboldt Foundation and thanks
    Heinrich-Heine-Universit\"at D\"usseldorf for its hospitality.  The
    third author was supported by the Spanish Government, grant
    MTM2011-28229-C02-02, partly FEDER funds, and by the Basque
    Government, grant IT-460-10.}

  \keywords{Pro-$p$ groups, Hausdorff dimension, 
      filtration series, $p$-adic analytic groups, soluble groups}

  \subjclass[2010]{Primary 20E18; Secondary 20F16, 20G25, 22E20, 28A78}


\begin{abstract} 
  Let $G$ be a finitely generated pro-$p$ group, equipped with the
  $p$-power series $\mathcal{P} \colon G_i = G^{p^i}$,
  $i \in \mathbb{N}_0$. The associated metric and
  Hausdorff dimension function
    $\hdim_G^\mathcal{P} \colon \{ X \mid X \subseteq G \} \to [0,1]$
  give rise to
  \[
  \hspec^\mathcal{P}(G) = \{ \hdim_G^\mathcal{P}(H) \mid H \le G \}
  \subseteq [0,1],
  \]
  the Hausdorff spectrum of closed subgroups of~$G$.  In
  the case where $G$ is $p$-adic analytic, the Hausdorff dimension
  function is well understood; in particular, $\hspec^\mathcal{P}(G)$
  consists of finitely many rational numbers closely
    linked to the analytic dimensions of subgroups of~$G$.

  Conversely, it is a long-standing open question
  whether $\lvert \hspec^\mathcal{P}(G) \rvert < \infty$ implies that
  $G$ is $p$-adic analytic.  We prove that the answer is yes, in a
  strong sense, under the extra condition that $G$ is soluble.

  Furthermore, we explore the problem and related questions also for
  other filtration series, such as the lower $p$-series, the Frattini
  series, the modular dimension subgroup series and
  quite general filtration series. For instance, we
    prove, for $p>2$, that every countably based
    pro-$p$ group $G$ with an open subgroup mapping onto
    $\mathbb{Z}_p \oplus \mathbb{Z}_p$ admits a filtration
    series $\mathcal{S}$ such that $\hspec^\mathcal{S}(G)$ contains an
    infinite real interval.
\end{abstract}

\maketitle


\section{Introduction}
The notion of Hausdorff dimension was pioneered by Hausdorff and
developed systematically by Besicovitch and others starting from the
1930s.  It is central to the subject of fractal geometry;
compare~\cite{Fa90}.  More recently, the concept of Hausdorff
dimension has led to fruitful applications in the context of profinite
groups; e.g.,
see~\cite{BaSh97,BaKl03,Er04,AbVi05,JaKl07,Er10,FeZu14,Gl14,FeGiGo}.
Let $G$ be a countably based profinite group and fix a
\emph{filtration series} $\mathcal{S}$ of $G$, i.e., a descending
chain $G = G_0 \supseteq G_1 \supseteq \ldots$ of open normal
subgroups $G_i \trianglelefteq G$ such that $\bigcap_i G_i = 1$.  Such
a chain forms a base of neighbourhoods of the identity and, if $G$ is
infinite, induces a translation-invariant metric on~$G$ such that the
distance between $x,y \in G$ is
$d^\mathcal{S}(x,y) = \inf \left\{ \lvert G : G_i \rvert^{-1} \mid x
  \equiv y \pmod{G_i} \right\}$.
This in turn yields the \emph{Hausdorff dimension}
$\hdim_G^\mathcal{S}(U) \in [0,1]$ of any subset $U \subseteq G$, with
respect to~$\mathcal{S}$.

Based on work of Abercrombie~\cite{Abercrombie}, Barnea and
Shalev~\cite{BaSh97} gave the following `algebraic' interpretation of
the Hausdorff dimension of a closed subgroup $H$ of $G$:
\begin{equation} \label{equ:hdim-formula} \hdim_G^\mathcal{S}(H) =
  \varliminf_{i\rightarrow \infty} \frac{\log \lvert HG_i : G_i
    \rvert}{\log \lvert G : G_i \rvert};
\end{equation}
we are interested in Hausdorff dimension as a density
  function in this sense and the formula motivates the convenient
backup definition
$\hdim_G^\mathcal{S}(H) = \log \lvert H \rvert / \log \rvert G \lvert$
for finite groups~$G$.  The \emph{Hausdorff spectrum} of $G$, with
respect to $\mathcal{S}$, is
\[
 \hspec^\mathcal{S}(G) = \{ \hdim_G^\mathcal{S}(H) \mid
 H \le  G\} \subseteq [0,1]
\]
and reflects the range of Hausdorff dimensions of closed subgroups $H$ of~$G$.

Throughout we will be concerned with pro-$p$ groups, where $p$ denotes
a fixed prime.  Even for comparatively well behaved groups, such as
$p$-adic analytic pro-$p$ groups~$G$, the Hausdorff dimension function
and the Hausdorff spectrum of $G$ are known to be sensitive to the
choice of~$\mathcal{S}$; see~\cite[Ex.~2.5]{BaSh97}.  However, as
emphasised in~\cite{BaSh97}, for every finitely generated pro-$p$
group $G$ there is a rather natural choice of filtration, namely the
\emph{$p$-power series}
\[
\mathcal{P} \colon \pi_i(G) = G^{p^i} = \langle x^{p^i} \mid x \in G
\rangle, \quad i \in \mathbb{N}_0.
\]
This is substantiated by the following result, which has recently been
generalised by Fern\'{a}ndez-Alcober, Giannelli and
Gonz\'{a}lez-S\'{a}nchez~\cite{FeGiGo} to cover Hausdorff
dimensions of analytic subgroups of $R$-analytic profinite groups over
general pro-$p$ domains~$R$.

\begin{theorem}[Barnea and
  Shalev~\cite{BaSh97}] \label{BarneaShalevThm} Let $G$ be an infinite
  $p$-adic analytic pro-$p$ group.  Then every closed subgroup
  $H \le G$ satisfies
  \[
  \hdim_G^\mathcal{P}(H) = \dim (H) / \dim(G), 
  \]
  where $\dim(X)$ denotes the analytic dimension of a $p$-adic
  manifold~$X$.

  In particular,
  $\hspec^\mathcal{P}(G) \subseteq \{0, \nicefrac{1}{d},\nicefrac{2}{d}, \ldots,
  \nicefrac{d-1}{d}, 1 \}$, where $d = \dim(G) \ge 1$.
\end{theorem}

It remains an open problem whether this theorem can actually be turned
into a characterisation of $p$-adic analytic pro-$p$ groups in terms
of Hausdorff spectra, in the spirit of~\cite[Interl.~A]{DDMS99}.
Contrary to what is perhaps suggested in the
introduction of~\cite{BaSh97}, one has to be rather careful, as the
assertion in Theorem~\ref{BarneaShalevThm} does not generally remain
valid with respect to other standard filtration series.  In passing,
we mention another relevant result from~\cite{BaSh97}, which
in turn relies on a theorem of Zelmanov~\cite{Ze92}: a
finitely generated pro-$p$ group $G$ is $p$-adic analytic if and only
if $G$ contains no infinite closed subgroup $H$ of Hausdorff dimension
$\hdim_G^\mathcal{P}(H) = 0$ with respect to the $p$-power
series~$\mathcal{P}$.

In this paper we consider, in addition to the $p$-power filtration,
three other natural and commonly used filtration series on finitely
generated pro-$p$ groups.  We recall that the \emph{lower $p$-series}
(sometimes called lower $p$-central series) of a finitely generated
pro\nobreakdash-$p$ group $G$ is given recursively by
\begin{align*}
  \mathcal{L} \colon P_1(G) = G,  %
  & \quad \text{and} \quad  P_i(G) = P_{i-1}(G)^p
    \, [P_{i-1}(G),G] \quad
    \text{for $i \geq 2$,} \\
  \intertext{while the \emph{Frattini series} of $G$ is
  given recursively by} 
  \mathcal{F} \colon \Phi_0(G) = G, %
  & \quad \text{and} \quad \Phi_i(G) = \Phi_{i-1}(G)^p
    \, [\Phi_{i-1}(G),\Phi_{i-1}(G)] \quad \text{for $i \geq 1$.} \\
  \intertext{The (modular) \emph{dimension subgroup series} (sometimes
  called Jennings or Zassenhaus series) of $G$ is closely
  related to the filtration of the group ring $\mathbb{F}_pG$ by powers
  of its augmentation ideal; it can be defined recursively
  by}
  \mathcal{D} \colon D_1(G) = G, %
  & \quad \text{and} \quad D_i(G) = D_{\lceil i/p \rceil}(G)^p 
    \prod_{1 \le j <i} [D_j(G),D_{i-j}(G)] \quad \text{for $i \geq 2$.}  
\end{align*}

Theorem~\ref{BarneaShalevThm} motivates a number of questions, related
to, but more specific than Problem~1 in~\cite{BaSh97}.

\begin{problem}
  Given any two filtration series $\mathcal{S}_1, \mathcal{S}_2$ of a
  finitely generated pro-$p$ group~$G$, is it true that
  $\hspec^{\mathcal{S}_1}(G)$ is finite if and only if
  $\hspec^{\mathcal{S}_2}(G)$ is finite?
\end{problem}

We solve this problem in the negative, by constructing for $p > 2$ a
tailor-made filtration series $\mathcal{S}$ of the abelian pro-$p$
group $G \cong \mathbb{Z}_p\oplus \mathbb{Z}_p$ such that
$\hspec^\mathcal{S}(G)$ contains an infinite real interval.  By
contrast, the Hausdorff spectrum
$\hspec^\mathcal{P}(G) = \{ 0, \nicefrac{1}{2}, 1 \}$ of the same
group with respect to the $p$-power series and indeed its spectra with
respect to other conventional filtration series are discrete; compare
Proposition~\ref{pro:close-to-p-power}.  More generally, we establish
the following result.

\begin{theorem} \label{thm:maps-onto-Zp2} Let $G$ be a countably based
  pro-$p$ group that has an open subgroup mapping
    surjectively onto $\mathbb{Z}_p \oplus \mathbb{Z}_p$.  Then there
  exists a filtration series $\mathcal{S}$ of $G$ such that
  $\hspec^\mathcal{S}(G)$ contains the complete real interval
  $[\nicefrac{1}{p+1}, \nicefrac{p-1}{p+1}]$.
\end{theorem}

This unexpectedly erratic behaviour motivates us to focus on Hausdorff
dimensions with respect to one of the four natural filtrations
$\mathcal{P}, \mathcal{L}, \mathcal{F}, \mathcal{D}$.

\begin{problems} \label{Problems} Let $G$ be a finitely generated
  pro-$p$ group, and let $\mathcal{P}, \mathcal{L}, \mathcal{F}, \mathcal{D}$
  denote the $p$-power series, the lower $p$-series, the Frattini
  series and the dimension subgroup series of~$G$.
  \begin{enumerate}
  \item[(1)] Suppose that
    $\lvert \hspec^\mathcal{S}(G) \rvert < \infty$ for at least one
    $\mathcal{S} \in \{ \mathcal{P}, \mathcal{L}, \mathcal{F},
    \mathcal{D} \}$.
    Does it follow that $\lvert \hspec^\mathcal{S}(G) \rvert < \infty$
    for all
    $\mathcal{S} \in \{ \mathcal{P}, \mathcal{L}, \mathcal{F},
    \mathcal{D} \}$?
  \item[(2)] Suppose that
    $\mathcal{S} \in \{ \mathcal{L}, \mathcal{F},
    \mathcal{D} \}$.
    If $G$ is $p$-adic analytic, does it follow that
    $\hspec^\mathcal{S}(G)$ is finite?  What is $\hspec^\mathcal{S}(G)$?
  \item[(3)] Suppose that $\mathcal{S} \in \{ \mathcal{P},
    \mathcal{L}, \mathcal{F}, \mathcal{D} \}$ and 
    $\lvert \hspec^\mathcal{S}(G) \rvert < \infty$.  Does it follow
    that $G$ is $p$-adic analytic?
  \end{enumerate}
\end{problems}

Regarding Problems~\ref{Problems}~(1) and (2) we obtain a positive
partial solution.

\begin{proposition} \label{pro:hdim-all-equal}
  Let $G$ be a $p$-adic analytic pro-$p$ group.  Then the Hausdorff
  dimension functions with respect to the $p$-power
  series~$\mathcal{P}$, the Frattini series~$\mathcal{F}$ and the
  dimension subgroup series~$\mathcal{D}$ coincide on closed
  subgroups~$H \le G$, i.e.,
  \[
  \hdim_G^\mathcal{P}(H) = \hdim_G^\mathcal{F}(H) =
  \hdim_G^\mathcal{D}(H).
  \]
  Consequently, $\hspec^\mathcal{P}(G) = \hspec^\mathcal{F}(G) =
  \hspec^\mathcal{D}(G)$.
\end{proposition}

The situation for the lower $p$-series, even for $p$-adic analytic
pro-$p$ groups, is less clear.  In Example~\ref{exa:lower-p-series} we
illustrate that there are $p$-adic analytic pro-$p$ groups $G$ for
which $\hspec^\mathcal{P}(G) \not = \hspec^\mathcal{L}(G)$.  In fact,
we provide a concrete family of infinite $p$-adic analytic pro-$p$
groups $G$ such that $\lvert \hspec^\mathcal{L}(G) \rvert / \dim(G)$
is unbounded as $\dim(G) \to \infty$ within this family.  This raises
yet more interesting questions.

\begin{problems} 
  Does there exist, for any given $p$-adic analytic pro-$p$ group $G$,
  a uniform bound $b(G)$ for $\lvert \hspec^{\mathcal{S}}(G) \rvert$,
  as $\mathcal{S}$ runs through all filtration series of $G$ with
  $\lvert \hspec^{\mathcal{S}}(G) \rvert < \infty$?  

  If yes, does there exist, for any $n \in \mathbb{N}$, a uniform
  bound $b(n)$ for $b(G)$, as $G$ runs through all $p$-adic analytic
  pro-$p$ groups of dimension $\dim(G) \leq n$?
\end{problems}

Regarding Problem~\ref{Problems}~(3) we obtain the following positive
partial solution, in a strong sense.  This gives a new
characterisation of $p$-adic analytic pro-$p$ groups amongst soluble
pro\nobreakdash-$p$ groups in terms of Hausdorff dimension.

\begin{theorem} \label{thm:soluble-interval} Let $G$ be a finitely
  generated soluble pro-$p$ group, and let $\mathcal{S}$ be any one of
  the $p$-power series~$\mathcal{P}$, the Frattini
  series~$\mathcal{F}$ or the dimension subgroup series~$\mathcal{D}$.
  If $G$ is not $p$-adic analytic then the Hausdorff spectrum
  $\hspec^\mathcal{S}(G)$ with respect to $\mathcal{S}$
  contains an infinite real interval.
\end{theorem}

\begin{corollary} \label{cor:characterisation-soluble} Let $G$ be a
  finitely generated soluble pro-$p$ group, and let $\mathcal{S}$ be
  one of $\mathcal{P}$, $\mathcal{F}$ or $\mathcal{D}$.  Then $G$ is
  $p$-adic analytic if and only if $\hspec^\mathcal{S}(G)$ is finite.
\end{corollary}

It remains an open question whether a statement similar to
Theorem~\ref{thm:soluble-interval} holds true for the lower
$p$-series.

Finally, we verify that Barnea and Shalev's characterisation of
$p$-adic analytic pro-$p$ groups in terms of Hausdorff
dimension~\cite[Th.~1.3]{BaSh97} remains partly valid with respect to
the other filtration series considered here.
 
\begin{theorem} \label{thm:Zelmanov-characterisation} Let $G$ be a
  finitely generated pro-$p$ group, and let $\mathcal{S}$ be any one
  of the $p$-power series~$\mathcal{P}$, the Frattini
  series~$\mathcal{F}$ or the dimension subgroup series~$\mathcal{D}$.
  Then the following are equivalent.
  \begin{enumerate}
  \item The group $G$ is $p$-adic analytic.
  \item There exists a constant $c \in (0,1]$ such that every infinite
    closed subgroup $H \leq G$ satisfies
    $\hdim_G^\mathcal{S}(H) \geq c$.
  \item Every infinite closed subgroup $H \leq G$ satisfies
    $\hdim_G^\mathcal{S}(H) > 0$.
  \item The group $G$ is finite, or there exists a closed subgroup $H \leq G$
    such that $H \cong \mathbb{Z}_p$ and $\hdim_G^\mathcal{S}(H) > 0$.
  \end{enumerate}
\end{theorem}

Again, the situation for the lower $p$-series is less clear.  In
Example~\ref{exa:two-SL2s} we illustrate, for
$\mathcal{S} = \mathcal{L}$, that (iv) does not generally imply (i),
(ii) or~(iii).  On the other hand, we establish in
Proposition~\ref{pro:lower-bound-for-lower-p}, for
$\mathcal{S} = \mathcal{L}$, that (i) still implies (ii), (iii)
and~(iv).

\medskip

\noindent \textit{Notation and Organisation.}  Throughout,
$\varliminf a_i = \varliminf_{i \to \infty} a_i$ denotes the lower
limit (limes inferior) of a sequence $(a_i)_{i \in \mathbb{N}}$ in
$\mathbb{R} \cup \{ \pm \infty \}$.  At times, some terms $a_i$ may
evaluate to $y/0$ for some $y \in \mathbb{R}_0$.  For definiteness, we
agree that $0/0 = 1$ and $y/0 = \infty$ for $y > 0$.  Tacitly,
subgroups of profinite groups are generally understood to be closed
subgroups.  As a default we set $P_0(G) = D_0(G) = G$ for every
finitely generated pro-$p$ group~$G$.

Theorem~\ref{thm:maps-onto-Zp2} is proved in Section~\ref{sec:3}.
Proposition~\ref{pro:hdim-all-equal} and
Theorem~\ref{thm:Zelmanov-characterisation} are established in
Section~\ref{sec:4}.  Theorem~\ref{thm:soluble-interval} is proved in
Section~\ref{sec:5}.
  

\section{Preliminaries}

In this section we collect some auxiliary results for later use.

\begin{proposition} \label{pro:shifting-hspec} Let
  $\varphi \colon G \rightarrow \widetilde{G}$ be an epimorphism
  between infinite countably based profinite groups, and let
  $\widetilde{\mathcal{S}} \colon \widetilde{G}_0 \supseteq
  \widetilde{G}_1 \supseteq \ldots$
  be a filtration series of~$\widetilde{G}$.  Then there exists a
  filtration series
  $\mathcal{S} \colon G_0 \supseteq G_1 \supseteq \ldots$ of~$G$ such
  that
  \[
  G_i \varphi = \widetilde{G}_i \quad \text{for all $i \in
  \mathbb{N}_0$,} \qquad  \text{and} \qquad 
  \hspec^\mathcal{S}(G) =
  \hspec^{\widetilde{\mathcal{S}}}(\widetilde{G}).
  \]
\end{proposition}

\begin{proof}
  Clearly we can find a filtration series
  $G = G_0^* \supseteq G_1^* \supseteq \ldots$ of $G$ such that
  $G_i^* \varphi \subseteq \widetilde{G}_i$ for all
  $i \in \mathbb{N}_0$.  The function
  \[
  f \colon \mathbb{N}_0 \rightarrow \mathbb{N}_0, \quad f(i) = \max \{
  j \in \mathbb{N}_0 \mid \widetilde{G}_i \subseteq G_j^* \varphi \}
  \]
  is non-decreasing and $\lim_{i \to \infty} f(i) = \infty$.  Choose,
  for each $i \in \mathbb{N}_0$, a suitable subset
  $X_i \subseteq G_{f(i)}^*$ such that
  $(\langle X_i \rangle^G \, G_i^*) \varphi = \widetilde{G}_i$.  By
  putting
  $G_i^{**} = \langle X_i \cup X_{i+1} \cup \ldots \rangle^G \,  G_i^*$,
  $i \in \mathbb{N}_0$, we construct a filtration series
  $G = G_0^{**} \supseteq G_1^{**} \supseteq \ldots$ of $G$ such that
  $G_i^{**} \varphi = \widetilde{G}_i$ for all $i \in \mathbb{N}_0$.

  Set $K = \mathrm{ker}(\varphi) \trianglelefteq G$.  By putting
  $K_i = K \cap G_i^{**}$, $i \in \mathbb{N}_0$, we obtain a
  filtration series of~$K$ that consists of normal subgroups of~$G$.
  Choosing a non-decreasing function
  $h \colon \mathbb{N}_0 \rightarrow \mathbb{N}_0$ such that
  \[
  \lim_{i \to \infty} h(i) = \infty \qquad \text{and} \qquad
  \frac{\log \lvert K : K_{h(i)} \rvert}{\log \lvert \widetilde{G} :
    \widetilde{G}_i \rvert} \underset{i \to \infty}{\longrightarrow}
  0,
  \]
  we define the filtration series
  \[
  \mathcal{S} \colon G_i = K_{h(i)} G_i^{**}, \quad i \in \mathbb{N}_0,
  \]
  of~$G$.  The filtration series $\mathcal{S}$ satisfies
  \[
  \frac{\log \lvert K G_i : G_i \rvert}{\log \lvert G : K G_i \rvert}
  = \frac{\log \lvert K : K \cap G_i \rvert}{\log \lvert G : K G_i
    \rvert} \le \frac{\log \lvert K : K_{h(i)} \rvert}{\log \lvert
    \widetilde{G} : \widetilde{G}_i \rvert} \underset{i \to
    \infty}{\longrightarrow} 0.
  \]
  Since
  \[
  \frac{\log \lvert K G_i : G_i \rvert}{\log \lvert G : G_i \rvert}
  \le \frac{\log \lvert K G_i : G_i \rvert}{\log \lvert G : K G_i
    \rvert},
  \]
  this ensures, in particular, that
  $\hdim_G^\mathcal{S}(K) = 0$ is given by a proper limit, that is
  \[
  \hdim_G^\mathcal{S}(K) = \lim_{i \to \infty} \frac{\log \lvert K G_i
    : G_i \rvert}{\log \lvert G : G_i \rvert} = 0.
  \]
  Furthermore, for every $H \le G$ we obtain
  \begin{align*}
    \hdim_G^\mathcal{S}(H) %
    & = \varliminf \frac{\log \lvert H G_i : G_i \rvert}{\log \lvert
      G : G_i \rvert} \\
    & = \varliminf \frac{\log \lvert H G_i : (H G_i \cap K G_i)
      \rvert}{\log \lvert G : KG_i \rvert + \log \lvert KG_i : G_i
      \rvert}  + \frac{\log \lvert (H G_i 
      \cap K G_i) : G_i \rvert}{\log \lvert  G : G_i \rvert} \\
    & = \varliminf \, \frac{\log \lvert HK G_i : KG_i \rvert / \log \lvert G :
      KG_i \rvert}{1 + \underbrace{\log \lvert K G_i : G_i \rvert / \log \lvert G : K
      G_i \rvert}_{\to 0 \text{ as } i \to \infty}} +
      \underbrace{\frac{\log \lvert (H G_i \cap K G_i) : G_i \rvert}{\log \lvert G
      : G_i \rvert}}_{\to 0  \text{ as } i \to \infty} \\ 
    & = \hdim_{\widetilde{G}}^{\widetilde{\mathcal{S}}}(H \varphi).
  \end{align*}  
  We conclude that
  \[
  G_i \varphi = \widetilde{G}_i \quad \text{for all $i \in
  \mathbb{N}_0$,} \qquad  \text{and} \qquad 
  \hspec^\mathcal{S}(G) =
  \hspec^{\widetilde{\mathcal{S}}}(\widetilde{G}). \qedhere
  \]
\end{proof}

The following lemma can be verified by routine arguments; e.g.,
see~\cite[Lem.~7.1]{Kl99}.

\begin{lemma} \label{lem:DPhil}
 Let $G$ be a countably based profinite group, and let
 \[
 \mathcal{S} \colon G = G_0 \supseteq G_1 \supseteq \ldots, \qquad
 \mathcal{S}^* \colon G = G_0^* \supseteq G_1^* \supseteq \ldots
 \]
 be filtration series of~$G$.  Suppose that
 \[
 \lim_{i \to \infty} \frac{\log \lvert G_i G_i^* : G_i \rvert}{\log
   \lvert G : G_i G_i^* \rvert} = \lim_{i \to \infty} \frac{\log
   \lvert G_i G_i^* : G_i^* \rvert}{\log \lvert G : G_i G_i^* \rvert}
 = 0.
 \]
 Then every closed subgroup $H \le G$ satisfies
 $\hdim_G^\mathcal{S}(H) = \hdim_G^{\mathcal{S}^*}(H)$.
\end{lemma}

\begin{corollary} \label{cor:2series}
  Let $G$ be a countably based profinite group, and let
  \[
  \mathcal{X} \colon G = X_0 \supseteq X_1 \supseteq \ldots, \qquad
  \mathcal{Y} \colon G = Y_0 \supseteq Y_1 \supseteq \ldots
  \]
  be filtration series of~$G$.  Suppose that
  \[
  \mathbb{N}_0 \to \mathbb{N}_0, \quad i \mapsto i^* \qquad \text{and}
  \qquad \mathbb{N}_0 \to \mathbb{N}_0, \quad j \mapsto j'
  \]
  are non-decreasing functions such that
  \[
  \lim_{i \to \infty} \frac{\log \lvert X_i Y_{i^*} : X_i
    \rvert}{\log \lvert G : X_i Y_{i^*} \rvert} = \lim_{i \to
    \infty} \frac{\log \lvert X_i Y_{i^*} : Y_{i^*} \rvert}{\log
    \lvert G : X_i Y_{i^*} \rvert} = 0
  \]
  and
  \[
  \lim_{j \to \infty} \frac{\log \lvert X_{j'} Y_j : X_{j'} \rvert}{\log
    \lvert G : X_{j'} Y_j \rvert} = \lim_{j \to \infty} \frac{\log
    \lvert X_{j'} Y_j : Y_j \rvert}{\log \lvert G : X_{j'} Y_j \rvert}
  = 0.
  \]
  Then every closed subgroup $H \le G$ satisfies
  $\hdim_G^\mathcal{X}(H) = \hdim_G^\mathcal{Y}(H)$.
\end{corollary}

\begin{proof}
  Observe that
  \[
  \mathcal{X}' \colon G = X_{0'} \supseteq X_{1' } \supseteq \ldots, \qquad
  \mathcal{Y}^* \colon G = Y_{0^*} \supseteq Y_{1^*} \supseteq \ldots
  \]
  are filtration series of~$G$, obtained by thinning out $\mathcal{X}$
  and~$\mathcal{Y}$.  Thus Lemma~\ref{lem:DPhil} implies,
  for every closed subgroup $H \le G$,
  \[
  \hdim_G^\mathcal{X}(H) = \hdim_G^{\mathcal{Y}^*}(H) \geq
  \hdim_G^\mathcal{Y}(H) = \hdim_G^{\mathcal{X}'}(H) \geq
  \hdim_G^\mathcal{X}(H). \qedhere
  \]
\end{proof}


\section{Filtration series and Hausdorff spectra for $\mathbb{Z}_p
  \oplus \mathbb{Z}_p$} \label{sec:3}
 
In this section we prove Theorem~\ref{thm:maps-onto-Zp2}.  Observe
that the Hausdorff spectrum of a procyclic pro-$p$ group
$G \cong \mathbb{Z}_p$ is $\hspec(G) = \{0,1\}$, no matter which
filtration series one chooses.  In this section we study, nearly systematically, how the Hausdorff spectrum
$\hspec^\mathcal{S}(G)$ for the abelian pro-$p$ group
$G = \mathbb{Z}_p\oplus \mathbb{Z}_p$ depends on the chosen filtration
series~$\mathcal{S}$.  According to Theorem~\ref{BarneaShalevThm} we
have $\hspec^\mathcal{P}(G) = \{0, \nicefrac{1}{2}, 1\}$ with respect
to the $p$-power series~$\mathcal{P}$.

Writing $v_p \colon \mathbb{Z}_p \to \mathbb{Z} \cup \{\infty\}$ for
the standard $p$-adic valuation on $\mathbb{Z}_p$, we prove the
following quite flexible result.
  
\begin{lemma} \label{lem:general-hdim-formulae} Let
  $G = \mathbb{Z}_p \oplus \mathbb{Z}_p$, and let 
  \[
  \mathcal{S} \colon
  G_i = \langle (p^{a_i},z_i), (0,p^{b_i}) \rangle, \quad
  i \in \mathbb{N}_0,
  \]
  be an arbitrary filtration series of~$G$; this means that
  $a_i, b_i \in \mathbb{N}_0$ and $z_i \in \mathbb{Z}_p$ satisfy the
  conditions: $\mathrm{(i)}$ $(a_i)_{i \in \mathbb{N}_0}$ and
  $(b_i)_{i \in \mathbb{N}_0}$ are non-decreasing sequences
  starting at $a_0 = b_0 = 0$ and diverging to infinity; $\mathrm{(ii)}$ for
  each $i \in \mathbb{N}$,
  \[
  b_{i-1} \le v_p(z_i - p^{a_i - a_{i-1}}z_{i-1}).
  \]

  Let $H$ be a non-trivial, non-open closed subgroup of~$G$.  Then
  either there exists $\lambda \in \mathbb{Z}_p$ such that $H$ is an open
  subgroup of $\langle (1,\lambda)\rangle$ and
  \[
  \hdim_G^\mathcal{S}(H) = \varliminf_{i \to \infty} \; \max \left \{
    \frac{a_i}{a_i+b_i} ,
    1-\frac{v_p(z_i-p^{a_i}\lambda)}{a_i+b_i} \right \},
  \]
  or there exists $\mu \in p\mathbb{Z}_p$ such that $H$ is an open
  subgroup of $\langle (\mu,1) \rangle$ and
  \[
  \hdim_G^\mathcal{S}(H) = \varliminf_{i \to \infty} \; \max \left \{
    \frac{a_i}{a_i+b_i} ,
    1-\frac{v_p(p^{a_i}-z_i\mu)}{a_i+b_i} \right \}.
  \]
\end{lemma}
 
\begin{proof}
  Clearly, a non-trivial, non-open closed subgroup $H$ of $G$ has
  analytic dimension $\dim(H) = 1$ and is procyclic.  Consequently,
  $H$ is open and hence of finite index in a maximal procyclic
  subgroup.  The latter are the groups of the form
  $\langle (1, \lambda) \rangle$ with $\lambda \in \mathbb{Z}_p$ and
  $\langle (\mu,1) \rangle$ with $\mu \in p\mathbb{Z}_p$.  Since $G$
  is infinite, the Hausdorff dimension function is constant on
  commensurability classes of subgroups of~$G$, and we may assume that
  $H = \langle (1,\lambda) \rangle$ or $H = \langle (\mu,1) \rangle$
  for suitable $\lambda$ or $\mu$.

  Suppose first that $H = \langle (1,\lambda) \rangle$ with
  $\lambda \in \mathbb{Z}_p$.  Observe that
  $\log_p \lvert G : G_i \rvert = a_i + b_i$ for $i \in \mathbb{N}_0$.
  Furthermore, setting $d_i = \min\{b_i,v_p(z_i-p^{a_i}\lambda)\}$, we
  see that the group $HG_i$ is equal to
  $\langle (1,\lambda), (0,p^{d_i}) \rangle$ so that
  $\log_p \lvert G : HG_i \rvert = d_i$ and
  \begin{align*}
    \log_p \lvert HG_i : G_i \rvert & = \log_p \lvert G : G_i \rvert -
                                      \log_p \lvert G : HG_i \rvert \\
                                    & = a_i + b_i - \min \{b_i, v_p(z_i - p^{a_i} \lambda ) \} \\
                                    & =\max \{a_i, a_i + b_i -
                                      v_p(z_i-p^{a_i} \lambda) \}.
  \end{align*}
  The result follows from~\eqref{equ:hdim-formula}.

  Now suppose that $H = \langle (\mu,1) \rangle$ with
  $\mu \in p\mathbb{Z}_p$.  We argue in a similar way, using for
  $i \in \mathbb{N}_0$ that $\log_p \lvert G : G_i \rvert = a_i + b_i$
  and $\log_p \lvert G : HG_i \rvert = d_i$, where
  $d_i = \min \{ b_i+v_p(\mu), v_p(p^{a_i}-z_i\mu) \}$.
\end{proof}

Next we consider filtration series of
$G = \mathbb{Z}_p \oplus \mathbb{Z}_p$, whose terms come from a fixed
`apartment', corresponding to a decomposition of $\mathbb{Q}_p^{\, 2}$
into a direct sum of two lines.  Such filtration series are somewhat
close to the $p$-power series.

\begin{proposition} \label{pro:close-to-p-power} Let
  $G = \mathbb{Z}_p \oplus \mathbb{Z}_p$, and let
  \[
  \mathcal{S} \colon
  G_i = \langle (p^{a_i},0), (0,p^{b_i}) \rangle, \quad
  i \in \mathbb{N}_0,
  \]
  be a filtration series of~$G$; this means that
  $(a_i)_{i \in \mathbb{N}_0}$ and $(b_i)_{i \in \mathbb{N}_0}$ are
  non-decreasing integer sequences, starting at $a_0 = b_0 = 0$ and
  diverging to infinity.

  Then, writing $x_i = a_i/(a_i+b_i)$ for $i \in \mathbb{N}_0$ and
  putting
    \[
    \xi = \min \left \{ \varliminf_{i \to \infty} x_i, \,
      \varliminf_{i \to \infty} (1-x_i) \right \}, \quad \eta =\max
    \left \{ \varliminf_{i \to \infty} x_i, \, \varliminf_{i \to
        \infty} (1-x_i) \right \},
    \]
    \[
    \zeta = \varliminf_{i \to \infty} \, \max \left \{ x_i, 1-
      x_i\right \},
    \]
    we have
    \[
    \hspec^\mathcal{S}(G) = \{0,\xi , \eta, \zeta, 1\},
    \]
    where
    \begin{equation} \label{eq:restriction} 0 \le \xi \le \eta \le
      1-\zeta \le \nicefrac{1}{2} \qquad \text{or} \qquad 0\le \xi \le
      1- \zeta = 1 - \eta \le \nicefrac{1}{2}.
    \end{equation}
    In particular, $\hspec^\mathcal{S}(G)$ is discrete of size at
    most~$5$.
   
    Conversely, for any $\xi, \eta, \zeta \in [0,1]$ satisfying
    \eqref{eq:restriction} there exists a filtration $\mathcal{S}$ of
    the above form such that
    $\hspec^\mathcal{S}(G)=\{0,\xi, \eta, \zeta, 1 \}$.
\end{proposition}

\begin{proof}
  The trivial subgroup has Hausdorff dimension~$0$, and
  open subgroups have Hausdorff dimension $1$ in~$G$.  It remains to
  deal with non-trivial, non-open closed subgroups of~$G$.  By
  Lemma~\ref{lem:general-hdim-formulae}, it suffices to consider
  procyclic subgroups of the form $\langle (1,\lambda) \rangle$ and
  $\langle (\lambda,1) \rangle$ for $\lambda \in \mathbb{Z}_p$.
  Furthermore, Lemma~\ref{lem:general-hdim-formulae} yields
  \[
  \hdim_G^\mathcal{S}(\langle (1,0)\rangle) = \varliminf x_i, \qquad
  \hdim_G^\mathcal{S}(\langle (0,1) \rangle) = \varliminf (1-x_i)
  \]
  and for $\lambda \in \mathbb{Z}_p \setminus \{0\}$,
  \[
  \hdim_G^\mathcal{S}(\langle (1,\lambda) \rangle) =
  \hdim_G^\mathcal{S}(\langle (\lambda,1) \rangle) = \varliminf \,
  \max \{ x_i , 1-x_i \} = \zeta. 
  \]
  
  Thus $\hspec^\mathcal{S}(G) = \{0,\xi,\eta,\zeta,1\}$ with
  $\xi, \eta$ as defined in the statement of the proposition.  We need
  to show that \eqref{eq:restriction} holds.  Without loss of
  generality we may assume $\xi = \varliminf x_i$ and
  $\eta = \varliminf (1-x_i)$.  We observe that
  $\xi \le \eta \le \zeta$ and $\zeta \ge \nicefrac{1}{2}$.

  \medskip
  
  \noindent \underline{Case 1.} Suppose $\eta \le \nicefrac{1}{2}$.
  Whenever $1-x_i$ is close to $\eta$, then $x_i$ is close to
  $1-\eta$.  Thus $\max \{x_i, 1-x_i\}$ is close to
  $\max \{1-\eta,\eta\} = 1 - \eta$ infinitely often, and hence $\zeta \le 1-\eta$.
  This establishes
  $0 \le \xi \le \eta \le 1-\zeta \le \nicefrac{1}{2}$.
  
  \medskip

  \noindent \underline{Case 2.} Suppose $\eta > \nicefrac{1}{2}$.
  Whenever $x_i$ is close to $\xi$, then $1- x_i$ is close
  to $1-\xi$.  Thus $1-x_i$ is close to
  $1-\xi$ infinitely often, and hence $\eta \le 1 - \xi$.
  Whenever $1-x_i$ is close to $\eta$, then $x_i$ is close
  to $1-\eta$.  Thus $\max \{x_i, 1-x_i\}$ is close to
  $\max \{1-\eta,\eta\} = \eta$ infinitely often, and then $\zeta \le \eta$ implies $\zeta =
  \eta$.
  This establishes
  $0 \le \xi \le 1-\zeta = 1 - \eta \le \nicefrac{1}{2}$.
  
  \medskip

  For the converse statement, given any
    $\xi, \eta, \zeta \in [0,1]$ satisfying the first condition
    in~\eqref{eq:restriction}, we can choose non-decreasing integer
    sequences $(a_i)_{i\in \mathbb{N}_0}$ and
    $(b_i)_{i \in \mathbb{N}_0}$ with $a_0 = b_0 = 0$ such that
    $x_i = a_i/(a_i+b_i)$ satisfies
  \begin{equation*}
    x_i \to
    \begin{cases}
      \xi & \text{as $i \to \infty$ subject to $i \equiv_3 0$,} \\
      1-\eta & \text{as $i \to \infty$ subject to $i \equiv_3 1$,} \\
      \zeta & \text{as $i \to \infty$ subject to $i \equiv_3 2$.}
    \end{cases}
  \end{equation*}
  This yields $\hspec^\mathcal{S}(G) = \{0,\xi, \eta, \zeta ,1 \}$ for
  $\mathcal{S}:G_i=\langle (p^{a_i},0),(0,p^{b_i})\rangle$,
  $i \in \mathbb{N}_0$.

  Similarly, given any $\xi, \eta, \zeta \in [0,1]$ satisfying the
  second condition in~\eqref{eq:restriction}, we can arrange that
  \begin{equation*}
    x_i \to
    \begin{cases}
      \xi & \text{as $i \to \infty$ subject to $i \equiv_2 0$,} \\
      1-\zeta & \text{as $i \to \infty$ subject to $i \equiv_2 1$,}
    \end{cases}
  \end{equation*}
  and
  $\hspec^\mathcal{S}(G) = \{0,\xi, \zeta, 1 \} = \{0,\xi, \eta, \zeta
  ,1 \}$ for the corresponding filtration $\mathcal{S}$.
\end{proof}

As recorded in the introduction, it would be interesting to find out
just how large the finite Hausdorff spectra $\hspec^\mathcal{S}(G)$ of
a $p$-adic analytic pro-$p$ group $G$ can be in relation to the
analytic dimension~$\dim(G)$.  Regarding infinite Hausdorff spectra we
obtain, rather unexpectedly, the following result.

\begin{proposition} \label{pro:Zp2}
  Let $G = \mathbb{Z}_p \oplus \mathbb{Z}_p$.  Then
  \[
  \mathcal{S} \colon
  G_i = \langle (p^{a_i},i p^{a_i}), (0,p^{b_i}) \rangle, \quad
  i \in \mathbb{N}_0,
  \]
  where $a_0 = b_0 = 0$ and $a_i = p^i$, $b_i = p^{i+1}$ for $i
  \geq 1$, is a filtration series of~$G$ such that 
  \[
  [ \nicefrac{1}{p+1}, \nicefrac{p-1}{p+1} ] \subseteq \hspec^\mathcal{S}(G).
  \]
\end{proposition}

\begin{proof}
  Let $\nu \in [ \nicefrac{1}{p+1}, \nicefrac{p-1}{p+1} ]$.  Then
  \[
  f \colon \mathbb{N} \to \mathbb{N}, \quad
  f(m) = \lceil p^{m+1} - p^m (p+1)\nu - 1 \rceil
  \]
  is a strictly increasing function.  Set $\lambda_0=1$ and, for $j \ge 1$,
  define $\lambda_j \in \mathbb{N}$ by
  \[
  \lambda_j = \lambda_{j-1} + p^{f(\lambda_{j-1})}.
  \]
  Observe that, as $j \to \infty$, the $\lambda_j$ converge with
  respect to the $p$-adic topology to some $\lambda \in \mathbb{Z}_p$,
  and moreover
  $v_p(\lambda - \lambda_j) = v_p(\lambda_{j+1} - \lambda_j) =
  f(\lambda_j)$
  for $j \ge 0$.  We show that
  the closed subgroup $H_{\lambda} = \langle (1,\lambda) \rangle$ satisfies
  $\hdim_G^\mathcal{S}(H_\lambda) = \nu$.

  Lemma~\ref{lem:general-hdim-formulae} yields
  \[
  \hdim_G^\mathcal{S}(H_\lambda) =\varliminf \, \max \left\{
    \nicefrac{1}{p+1}, r_i \right\}, \qquad \text{where} \quad r_i =
  \frac{p^{i+1} - v_p(i-\lambda)}{p^i (p+1)}.
  \]
  For $j \ge 0$ we observe that
  \[
  r_{\lambda_j} = \frac{p^{ \lambda_j +1} - v_p(\lambda_j
    -\lambda)}{p^{\lambda_j} (p+1)} = \frac{p^{ \lambda_j +1} -
    f(\lambda_j)}{p^{\lambda_j} (p+1)} = \frac{p^{ \lambda_j +1} -
    \lceil p^{\lambda_j +1} - p^{\lambda_j} (p+1) \nu - 1
    \rceil}{p^{\lambda_j} (p+1)}.
  \]
  As $\lceil x-1 \rceil \le x \le \lceil x \rceil$ for any
  $x \in \mathbb{R}$, we obtain
  \[
  \nu \le r_{\lambda_j}\le \nu+\frac{1}{p^{\lambda_j}(p+1)}, \qquad
  \text{for $j \ge 0$},
  \]
  and thus 
  \[
  \hdim_G^\mathcal{S}(H_\lambda) \le \nu.
  \]

  It suffices to show that
  \[
  r_i\ge \nu \qquad \text{for all $i \in \mathbb{N} \smallsetminus \{
    \lambda_j \mid j \ge 0 \}$.}
  \]
  Let $i \in \mathbb{N}$ and $j \ge 0$ such that
  $\lambda_j < i< \lambda_{j+1}$.  Since $i < \lambda_{j+1}$, we may
  write $i$ as $i = i_0 + i_1 p+\ldots +i_s p^s$ in base~$p$, where
  $0 \le i_0, \ldots, i_s < p$ and $s \le f(\lambda_j)$.  By the
  construction of $\lambda$, this yields
  \[
  v_p(i-\lambda) \le f(\lambda_j) = \lceil
  p^{\lambda_j+1}-p^{\lambda_j}(p+1)\nu-1 \rceil\le
  p^{\lambda_j+1}-p^{\lambda_j}(p+1)\nu
  \]
  and thus
  \[
  r_i\ge \frac{p^{i+1}-(p^{\lambda_j+1} -
    p^{\lambda_j}(p+1)\nu)}{p^i(p+1)} =
  \frac{p-p^{\lambda_j-i+1}+p^{\lambda_j-i}(p+1)\nu}{p+1}.
  \]
  Using $\lambda_j < i$ we deduce that
  \[
 r_i\ge \frac{p-1+p^{\lambda_j-i}(p+1)\nu}{p+1}\ge \frac{p-1}{p+1}\ge
 \nu. \qedhere
 \]
\end{proof}

\begin{proof}[Proof of Theorem~\ref{thm:maps-onto-Zp2}]
  The assertion follows from Propositions~\ref{pro:Zp2}
  and~\ref{pro:shifting-hspec}.
\end{proof}


\section{Hausdorff dimension with respect to the lower $p$-series, the
  Frattini series and the dimension subgroup series} \label{sec:4}

In this section we establish Proposition~\ref{pro:hdim-all-equal} and
Theorem~\ref{thm:Zelmanov-characterisation}.  We begin, however, with
an example illustrating that Hausdorff dimension with respect to the
lower $p$-series is somewhat more delicate.

\begin{example} \label{exa:lower-p-series} Fix $m \in \mathbb{N}$ and
  consider the cyclotomic extension of the ring of $p$-adic integers,
  $\mathfrak{O} = \mathbb{Z}_p[\zeta]$, where $\zeta$ denotes a
  primitive $p^m$th root of unity.  Recall that the cyclotomic field
  $\mathbb{Q}_p(\zeta)$ is a totally ramified extension of
  $\mathbb{Q}_p$ of degree~$\varphi(p^m) = (p-1)p^{m-1}$.  Indeed,
  $\pi = \zeta -1$ is a uniformising element and
  \[
  \mathfrak{O}_+ = \mathbb{Z}_p \oplus \pi \mathbb{Z}_p \oplus \ldots
  \oplus \pi^{\varphi(p^m)-1} \mathbb{Z}_p \cong \mathbb{Z}_p^{\,
    \varphi(p^m)}.
  \]
  Furthermore, we have
  $\pi^{\varphi(p^m)} \mathfrak{O} = p \mathfrak{O}$;
  compare~\cite[IV, Prop.~17]{Se79}.

  We choose $d \in \mathbb{N}$ and form the semidirect product
  $G = T \ltimes A$, where
  \begin{itemize}
  \item $T = \langle s_0, s_1, \ldots, s_{d-1} \rangle \cong
    \mathbb{Z}_p^{\, d}$,
  \item
    $A = \langle a_0, \ldots, a_{\varphi(p^m)-1} \rangle \cong
    \mathfrak{O}_+$
    via $\psi \colon A \to \mathfrak{O}$,
    $\prod a_i^{\, \lambda_i} \mapsto \sum \lambda_i \pi^i$ and
  \item the action of $T$ on $A$ is given by
    \[
    (b^{s_0}) \psi = b \psi \cdot \zeta \quad \text{and} \quad [b,s_1] =
    \ldots = [b,s_{d-1}] = 1 \quad  \text{for $b \in A$.}
    \]
  \end{itemize}
  Clearly, $G$ has analytic dimension $\dim (G) = d + \varphi(p^m)$.
  It is straightforward to compute the terms of the lower $p$-series
  of~$G$:
  \[
  \mathcal{L} \colon P_i(G) = \langle s_0^{\, p^{i-1}},
  \ldots, s_d^{\, p^{i-1}} \rangle \ltimes A_i,
  \quad \text{where $A_i \psi = \pi^{i-1} \mathfrak{O}$.}
  \]
  In particular, this gives
  $\log_p \lvert P_i(G) : P_{i+1}(G) \rvert = d+1$.  We deduce, for
  instance,
  \[
  \hdim_G^\mathcal{L} ( \langle s_0 \rangle ) = \nicefrac{1}{d+1} \quad
  \text{and} \quad \hdim_G^\mathcal{L}(\langle a_0 \rangle) =
  \nicefrac{1}{(d+1) \varphi(p^m)}
  \]
  in contrast to Theorem~\ref{BarneaShalevThm}.  A routine
  verification yields
  \[
  \hspec^\mathcal{L}(G) = \big\{ \nicefrac{j}{(d+1) \varphi(p^m)} \mid j \in \{ 0,
  1, \ldots, (d+1) \varphi(p^m) \} \big\}.
  \]
  Indeed, one shows easily for any $H \leq G$ that
  \[
  \hdim_G^\mathcal{L}(H) = \nicefrac{1}{d+1} \left(
    \dim(HA/A) + \nicefrac{1}{\varphi(p^m)} \dim(H \cap A) \right).
  \]
  In particular, for $0 \leq j \leq d$ and
  $0 \leq k \leq \varphi(p^m)$ the abelian groups
  \[
    H_{j,k}  = \langle s_0^{\, p^m}, \ldots, s_{j-1}^{\, p^m}, a_0, \ldots, a_{k-1} \rangle
  \]
  have Hausdorff dimension
  \[
  \hdim_G^\mathcal{L}(H_{j,k}) = \nicefrac{1}{d+1}\, (j +
  \nicefrac{k}{\varphi(p^m)}).   
  \]

  We observe that 
  \[
  \frac{\lvert \hspec^\mathcal{L}(G) \rvert}{\dim(G)} = \frac{(d+1)
    \varphi(p^m) + 1}{d + \varphi(p^m)} \to d+1 \quad \text{as $m \to \infty$}
  \]
  is unbounded as $d$ tends to infinity.
\end{example}

Next we show that the Hausdorff dimension functions with respect to
the $p$-power series, the Frattini series and the dimension subgroup
series coincide on $p$-adic analytic groups.

\begin{proof}[Proof of Proposition~\ref{pro:hdim-all-equal}] Being
  $p$-adic analytic, the group $G$ has finite rank.  Let $H \le G$ be
  a closed subgroup, and let
  $\mathcal{P} \colon G_i = \pi_i(G) = G^{p^i}$, $i \in \mathbb{N}_0$,
  denote the $p$-power series of~$G$.

  First we compare $\mathcal{P}$ to the Frattini series
  $\mathcal{F} \colon \Phi_i(G)$, $i \in \mathbb{N}_0$.  Clearly,
  $G_i \subseteq \Phi_i(G)$ for all $i \in \mathbb{N}_0$.  On the other hand,
  \cite[Prop.~3.9 and Th.~4.5]{DDMS99} allow us to deduce that there exists
  $j \in \mathbb{N}_0$ such that $G_i$ and $\Phi_i(G)$ are uniformly
  powerful for all~$i \ge j$.  Writing $d = \dim(G)$, we deduce, in
  particular, that for $i \ge j$ there are $x_1, \ldots, x_d \in G$
  such that
  $G_i = \langle x_1^{\, p^i}, \ldots, x_d^{\, p^i} \rangle$, and
  consequently
  $G_{i+1} = \langle x_1^{\, p^{i+1}}, \ldots, x_d^{\, p^{i+1}}
  \rangle = G_i^{\, p}$.
  This gives
  $\log_p \lvert G_i : G_{i+1} \rvert = \dim(G) = \log_p \lvert
  \Phi_i(G) : \Phi_{i+1}(G) \rvert$ for all $i \ge j$, and thus
  \[
  \log_p \lvert \Phi_i(G) : G_i \rvert =  \log_p \lvert \Phi_j(G) :
  G_j \rvert
  \] 
  is constant for $i \ge j$.  Therefore Lemma~\ref{lem:DPhil} yields
  $\hdim_G^\mathcal{P}(H) = \hdim_G^\mathcal{F}(H)$.

  Next we compare $\mathcal{P}$ to the dimension subgroup series
  $\mathcal{D} \colon D_i(G)$, $i \in \mathbb{N}$.  Clearly,
  $G_i \subseteq D_i(G)$ for all $i \in \mathbb{N}$.  According to
  the argument used above and \cite[Th.~11.4, 11.5;
  Lem.~11.22]{DDMS99}, we find $j \in \mathbb{N}$ such that:
  $G_i$ and $D_i(G)$ are uniformly powerful, and
  $D_{pi}(G) = D_i(G)^p$ for all $i \ge j$.  Consider the thinned out
  filtration series $\mathcal{D}^* \colon D_i^*(G) = D_{p^i j}(G)$,
  $i \in \mathbb{N}_0$.
  Then $G_i \subseteq D_{i-j}^*$ for $i \ge j$, with
  $\log_p \lvert G_i : G_{i+1} \rvert = \dim(G) = \log_p \lvert
  D_{i-j}^*(G) : D_{i-j+1}^*(G) \rvert$ and thus
  \[
  \log_p \lvert D_{i-j}^*(G) : G_i \rvert =  \log_p \lvert D_0^*(G) :
  G_j \rvert
  \] 
  is constant.  We conclude from Lemma~\ref{lem:DPhil} that
  $\hdim_G^\mathcal{P}(H) = \hdim_G^{\mathcal{D}^*}\!(H)$.
  Furthermore, we note that
  $\log_p \lvert D_i^*(G) : D_k(G) \rvert \leq \dim(G)$
  for $p^i j \le k \le p^{i+1}j$.  Hence, associating each term
  $D_i^*(G)$ of the series $\mathcal{D}^*$ with the terms
  $D_{p^ij}(G)$, $D_{p^ij+1}(G)$, \ldots, $D_{p^{i+1}j-1}(G)$ of the
  series $\mathcal{D}$, we employ Corollary~\ref{cor:2series} to
  deduce that $\hdim_G^{\mathcal{D}^*}\!(H) = \hdim_G^\mathcal{D}(H)$.
\end{proof}


  In preparation for the proofs of
  Theorem~\ref{thm:Zelmanov-characterisation} in this section and
  Theorem~\ref{thm:soluble-interval} in the next section we establish
  an auxiliary result of general interest.

  \begin{proposition}\label{pro:p-adic-analytic-qu-0}
    Let $G$ be a finitely generated pro-$p$ group, and let
    $\mathcal{S}$ be any one of the $p$-power series~$\mathcal{P}$,
    the Frattini series~$\mathcal{F}$ or the dimension subgroup
    series~$\mathcal{D}$ of~$G$.  Let $N \trianglelefteq H \le G$ be
    closed subgroups.

    Suppose that $G$ is not $p$-adic analytic, but $H/N$ is $p$-adic
    analytic.  Then
    \[
    \hdim_G^\mathcal{S}(H) = \hdim_G^\mathcal{S}(N).
    \]
  \end{proposition}

  \begin{proof}
    For $i \in \mathbb{N}_0$ we denote the $i$th term of $\mathcal{S}$
    by~$G_i$ and the $i$th term of the corresponding filtration of $H$
    by $H_i$, i.e.,
    \[
    G_i =
    \begin{cases}
      \pi_i(G) = G^{p^i} & \text{if $\mathcal{S} = \mathcal{P}$,} \\
      \Phi_i(G) & \text{if $\mathcal{S} = \mathcal{F}$,} \\
      D_i(G) & \text{if $\mathcal{S} = \mathcal{D}$}
  \end{cases}
  \quad \text{and} \quad
    H_i =
    \begin{cases}
      \pi_i(H) = H^{p^i} & \text{if $\mathcal{S} = \mathcal{P}$,} \\
      \Phi_i(H) & \text{if $\mathcal{S} = \mathcal{F}$,} \\
      D_i(H) & \text{if $\mathcal{S} = \mathcal{D}$.}
  \end{cases}
  \]
 
  We observe that, for $i \in \mathbb{N}_0$,
  \begin{equation}\label{equ:split-the-sum}
    \frac{\log_p \lvert H G_i : G_i \rvert}{\log_p \lvert G : G_i
      \rvert} = \frac{\log_p \lvert H G_i : N G_i \rvert}{\log_p \lvert
      G : G_i \rvert} + \frac{\log_p \lvert N G_i : G_i \rvert}{\log_p
      \lvert G : G_i \rvert}
  \end{equation}
  and
  \[
  \lvert H G_i : N G_i \rvert = \lvert H/N : (H \cap G_i)N/N \rvert
  \leq \lvert H/N : H_iN/N \rvert.
  \]
  Below we show that
  \begin{equation}\label{equ:tends-to-0}
    \frac{\log_p \lvert H /N : H_i N/N \rvert}{\log_p \lvert G : G_i
      \rvert} \underset{i \to \infty}{\longrightarrow} 0,
  \end{equation}
  and hence, taking lower limits in~\eqref{equ:split-the-sum}, we
  obtain
  \[
  \hdim_G^\mathcal{S}(H) = 0 + \varliminf \frac{\log_p \lvert
    N G_i : G_i \rvert}{\log_p \lvert G : G_i \rvert} =
  \hdim_G^\mathcal{S}(N).
  \]

  Thus it remains to establish~\eqref{equ:tends-to-0}.  Replacing $H$ by an
  open subgroup, if necessary, we may assume that $H/N$ is uniformly
  powerful.  Let $d = d(H/N)$ denote the minimal number of generators
  of~$H/N$.

  First suppose $\mathcal{S} = \mathcal{P}$.  Since $H/N$ is uniformly
  powerful, we conclude that $\log_p \lvert H/N : H_iN/N \rvert = di$
  for all $i \in \mathbb{N}_0$.  Furthermore, a characterisation of
  $p$-adic analytic pro-$p$ groups due to Lazard implies that
  $\log_p \lvert G : G_i \rvert \geq p^i$ for all
  $i \in \mathbb{N}_0$; see \cite[App.~A.3~(3.8.3)]{La65} and compare
  \cite[Cor.\ 11.6 and 11.19]{DDMS99}.  This yields
  \[
  \frac{\log_p \lvert H/N : H_iN/N \rvert}{\log_p \lvert G : G_i
    \rvert} \leq \frac{di}{p^i} \underset{i \to
    \infty}{\longrightarrow} 0.
  \]

  Next suppose $\mathcal{S} = \mathcal{F}$.  By
  \cite[Prop.~3.9]{DDMS99}, for $r \in \mathbb{N}$ there exists
  $i(r) \in \mathbb{N}$ such that for all $i \in \mathbb{N}$ with
  $i \geq i(r)$ we have
  $\log_p \lvert G_i : G_{i+1} \rvert = d(G_i) \geq r$; here $d(G_i)$
  denotes the minimal number of generators of~$G_i$.  On the other
  hand, $H^{p^i} \subseteq H_i$ and hence
  $\log_p \lvert H/N : H_iN/N \rvert \leq di$ for all
  $i \in \mathbb{N}_0$.  This implies
  \[
  \frac{\log_p \lvert H/N : H_iN/N \rvert}{\log_p \lvert G :
    G_i \rvert} \leq \frac{di}{(i-i(r))r}
  \underset{i \to \infty}{\longrightarrow} \nicefrac{d}{r}
  \]
  and, letting $r \to \infty$, we conclude that \eqref{equ:tends-to-0}
  holds.    
 
  Finally suppose $\mathcal{S} = \mathcal{D}$.  For $i \in \mathbb{N}$
  set $k(i) = \lfloor \log_p i \rfloor \in \mathbb{N}_0$ so that
  $p^{k(i)} \leq i < p^{k(i)+1}$.  By \cite[Th.\ 11.4]{DDMS99}, we
  have $\log_p \lvert G : G_i \rvert \geq i-1$.  On the other hand,
  $H^{p^{k(i)+1}} \subseteq H_{p^{k(i)+1}} \subseteq H_i$ and
  hence $\log_p \lvert H/N : H_iN/N \rvert \leq d(k(i)+1)$.  This
  implies
  \[
  \frac{\log_p \lvert H/N : H_iN/N \rvert}{\log_p \lvert G : G_i
    \rvert} \leq \frac{d(k(i)+1)}{i-1} \underset{i \to
    \infty}{\longrightarrow} 0. \qedhere
  \]
  \end{proof}


Before moving on to the proof of
Theorem~\ref{thm:Zelmanov-characterisation}, we remark that for
$\mathcal{S} = \mathcal{P}$ the statement was already proved (though
stated in a different way) by Barnea and Shalev~\cite{BaSh97}, using a
result of Zelmanov~\cite{Ze92}.  The latter implies that any infinite
finitely generated pro-$p$ group contains an element of infinite order
and hence a subgroup isomorphic to~$\mathbb{Z}_p$.

\begin{proof}[Proof of Theorem~\ref{thm:Zelmanov-characterisation}]
  First we show that (i) implies (ii).  Suppose that $G$ is $p$-adic
  analytic.  Observe that a closed subgroup $H \leq G$ is infinite if
  and only if it has analytic dimension $\dim(H) \geq 1$.  Thus
  Theorem~\ref{BarneaShalevThm} and
  Proposition~\ref{pro:hdim-all-equal} imply that (ii) holds,
  with $c = \nicefrac{1}{\dim(G)}$ if $G$ is infinite.

  Obviously (ii) implies (iii), and (iii) implies (iv) by Zelmanov's
  result~\cite{Ze92}.  It remains to show that (iv) implies~(i).
  Arguing by contraposition, we suppose that $G$ is not $p$-adic
  analytic.  Then $G$ is infinite.  Let $H \leq G$ be such
    that $H \cong \mathbb{Z}_p$.  Then
    Proposition~\ref{pro:p-adic-analytic-qu-0}, for $N=1$,
    implies~$\hdim_G^\mathcal{S}(H) = 0$.
\end{proof}

The following example illustrates that extending
Theorem~\ref{thm:Zelmanov-characterisation} to the lower $p$-series
requires more care: for $\mathcal{S} = \mathcal{L}$,
condition (iv) does not generally imply (i), (ii) or~(iii).

\begin{example} \label{exa:two-SL2s} 
  Consider $G = G_1 \times G_2$, where
  $G_1 = \mathrm{SL}_3^1(\mathbb{F}_p[\![t]\!])$ and
  $G_2 = \mathrm{SL}_3^1(\mathbb{Z}_p)$.  Observe that the lower
  $p$-series $\mathcal{L}$ of $G$ satisfies
  \[
  P_i(G) = P_i(G_1) \times P_i(G_2) =
  \mathrm{SL}_3^i(\mathbb{F}_p[\![t]\!]) \times
  \mathrm{SL}_3^i(\mathbb{Z}_p) \quad \text{for $i \in \mathbb{N}$},
  \]
 where
    \[
    \mathrm{SL}_3^i(\mathbb{F}_p[\![t]\!]) = \{ g \in
    \mathrm{SL}_3(\mathbb{F}_p[\![t]\!]) \mid g \equiv_{t^i} 1 \}
    \quad \text{and} \quad \mathrm{SL}_3^i(\mathbb{Z}_p) = \{ g \in
    \mathrm{SL}_3(\mathbb{Z}_p) \mid g \equiv_{p^i} 1 \}
    \]
    denote the $i$th principal congruence subgroups of
    $\mathrm{SL}_3(\mathbb{F}_p[\![t]\!])$ and
    $\mathrm{SL}_3(\mathbb{Z}_p)$; compare~\cite[Prop.~13.29]{DDMS99}.
    Thus $\log_p \lvert P_i(G) : P_{i+1}(G) \rvert = 8+8 = 16$ for all
    $i \in \mathbb{N}$.  

    For $j \in \{1,2\}$ choose $H_j \leq G_j$ with
    $H_j \cong \mathbb{Z}_p$.  The $p$-power maps in
    $\mathrm{SL}_3(\mathbb{F}_p[\![t]\!])$ and
    $\mathrm{SL}_3(\mathbb{Z}_p)$ behave rather differently: for
    $g \in \mathrm{SL}_3^i(\mathbb{F}_p[\![t]\!])$ we have
    $g^p \in \mathrm{SL}_3^{pi}(\mathbb{F}_p[\![t]\!])$, whereas for
    $g \in \mathrm{SL}_3^i(\mathbb{Z}_p) \smallsetminus
    \mathrm{SL}_3^{i+1}(\mathbb{Z}_p)$
    we have
    $g^p \in \mathrm{SL}_3^{i+1}(\mathbb{Z}_p) \smallsetminus
    \mathrm{SL}_3^{i+2}(\mathbb{Z}_p)$;
    compare \cite[Sec.~4]{BaSh97} and \cite[Prop.~13.22]{DDMS99}. This
    yields
    \[
    \frac{\log_p \lvert H_1 P_{i+1}(G) : P_{i+1}(G) \rvert}{\log_p \lvert
    G : P_{i+1}(G) \rvert} \le \frac{\lfloor \log_p i \rfloor +1}{16 i} \underset{i \to
    \infty}{\longrightarrow} 0
    \]
    and hence
    \[
    \hdim_G^\mathcal{L}(H_1) = 0, \quad \text{whereas} \quad
    \hdim_G^\mathcal{L}(H_2) = \nicefrac{1}{16}. \qedhere
    \]
\end{example}

\begin{proposition} \label{pro:lower-bound-for-lower-p} Let $G$ be
  an infinite $p$-adic analytic pro-$p$ group.  Then
  every closed subgroup $H \leq G$ satisfies
  $\hdim_G^\mathcal{L}(H) \geq \nicefrac{\dim(H)}{\dim(G)^2}$ with
  respect to the lower $p$-series~$\mathcal{L}$.
\end{proposition}

\begin{proof}
  Set $d = \dim(G)$ and choose a uniformly powerful open normal
  subgroup $U \trianglelefteq G$.  Let $j \in \mathbb{N}$ be such that
  $G^{p^j} \subseteq P_{j}(G) \subseteq U$.  Writing $G_i = G^{p^i}$
  and $U_i = U^{p^i}$ for the terms of the $p$-power filtrations of
  $G$ and~$U$, we observe that this implies
  \[
  G_i \subseteq U_{i-j} \quad \text{and} \quad P_i(G) \subseteq
  U_{\lfloor (i-j)/d \rfloor} \quad \text{for all $i \in \mathbb{N}$
    with $i \geq j$.}
  \]
  
  Let $H \leq G$ be any closed subgroup; without loss of
  generality we may assume that $H \leq U$.  Then there are
  constants $c_1, c_2 \in \mathbb{N}$ such that for all $i \in
  \mathbb{N}$ with $i \geq j$,
  \[
  \log_p \lvert H P_i(G) : P_i(G) \rvert \geq \log_p \lvert H
  U_{\lfloor (i-j)/d \rfloor} : U_{\lfloor (i-j)/d \rfloor} \rvert
  \geq \dim(H)\cdot\lfloor (i-j)/d \rfloor - c_1
  \]
  and
  \[
  \log_p \lvert G : P_i(G) \rvert = \log_p \lvert G : P_j(G) \rvert +
  \log_p \lvert P_j(G) : P_i(G) \rvert \leq c_2 + d(i-j).
  \]
  This gives
  \[
  \frac{\log_p \lvert H P_i(G) : P_i(G) \rvert}{\log_p \lvert G :
    P_i(G) \rvert} \geq \frac{\dim(H)\cdot \lfloor (i-j)/d \rfloor - c_1}{c_2 + d(i-j)}    \underset{i
    \to \infty}{\longrightarrow} \nicefrac{\dim(H)}{d^2},
  \]
  and we conclude that $\hdim_G^\mathcal{L}(H) \geq \nicefrac{\dim(H)}{d^2}$.
\end{proof}

Perhaps Proposition~\ref{pro:lower-bound-for-lower-p} can
  serve as a first step toward a positive solution of
  Problem~\ref{Problems}~(2) for $\mathcal{S} = \mathcal{L}$.


\section{Characterisation of soluble $p$-adic analytic
  groups} \label{sec:5} 

In this section we prove Theorem~\ref{thm:soluble-interval}
and obtain as an immediate consequence
  Corollary~\ref{cor:characterisation-soluble}.

\begin{lemma} \label{lem:xyz-inequalities} Let
  $x,y,z \in \mathbb{R}_{>0}$ and $\eta \in [0,1]$.

  \begin{enumerate}
  \item[(i)] If $\nicefrac{x}{y}\ge \eta-\nicefrac{1}{y}$, then
    $\nicefrac{x+z}{y+z}\ge \eta - \nicefrac{1}{y+z}$.
  \item[(ii)] If $\nicefrac{x}{y}\ge \eta$, then
    $\nicefrac{x+z}{y+z}\ge \eta$.
  \end{enumerate}
\end{lemma}

\begin{proof}
  Indeed, $\nicefrac{x}{y}\ge \eta - \nicefrac{1}{y}$ implies
  $x\ge \eta y -1$ and, using $\eta \le 1$, we deduce that
  $x+z\ge \eta y+ \eta z -1$.  This gives
  $\nicefrac{x+z}{y+z}\ge \eta - \nicefrac{1}{y+z}$.  

  Similarly, $\nicefrac{x}{y}\ge \eta$ implies $x \ge \eta y$ and,
  using $\eta \le 1$, we deduce that $x+z \ge \eta y + \eta z$.  This
  gives $\nicefrac{x+z}{y+z}\ge \eta$.
\end{proof}

\begin{proposition} \label{pro:fin-gen-less-eta} Let $G$ be a
  countably based pro-$p$ group which is not finitely generated, and
  let $\mathcal{S} \colon G_0 \supseteq G_1 \supseteq \ldots$ be a
  filtration series of~$G$.  Suppose that $\eta \in [0,1]$ is such that every
  finitely generated subgroup $H$ of $G$ satisfies
  $\hdim_G^\mathcal{S}(H) \le \eta$.  Then there exists a closed subgroup
  $H \le G$ such that $\hdim_G^\mathcal{S}(H) = \eta$.
\end{proposition}

\begin{proof} 
  Suppose that $\hdim_G^\mathcal{S}(H) < \eta$ for
  every finitely generated subgroup $H$ of~$G$.  In particular this
  implies that~$\eta>0$.  Furthermore, we may assume that
  $G_i \supsetneq G_{i+1}$ for all $i \in \mathbb{N}_0$.  Thus
  \[
  m(i) = \log_p \lvert G : G_i \rvert, \quad i \in
  \mathbb{N}_0,
  \]
  is a strictly increasing sequence in~$\mathbb{N}_0$, and $G/G_i$ is
  a finite $p$-group of order $p^{m(i)}$ for each~$i \in \mathbb{N}$.

  Below we construct recursively an ascending sequence of finitely
  generated closed subgroups
  $1 = H_0 \subseteq H_1 \subseteq H_2 \subseteq \ldots$ such that
  \begin{enumerate}
  \item $H_i G_i = H_{i'} G_i$ for all $i,i' \in \mathbb{N}$ with $i\le i'$;
  \item for all $i \in \mathbb{N}$,
    \[
    \tfrac{1}{m(i)} \log_p \lvert H_i G_i : G_i \rvert \ge \eta -
    \nicefrac{1}{m(i)};
    \]
  \item for infinitely many $i \in \mathbb{N}$,
    \begin{equation} \label{equ:inf-many} \tfrac{1}{m(i)} \log_p
      \lvert H_i G_i : G_i \rvert \le \eta.
    \end{equation}
  \end{enumerate}
 Setting $H = \langle H_0 \cup H_1 \cup \ldots \rangle\le G$, we obtain
 \begin{equation*}
   \hdim_G^\mathcal{S}(H) = \varliminf\,
   \tfrac{1}{m(i)} \log_p \lvert HG_i:G_i \rvert
   = \varliminf \, \tfrac{1}{m(i)} \log_p \lvert H_iG_i:G_i
   \rvert 
   = \eta.
 \end{equation*}

 Thus it remains to construct $H_1\subseteq H_2 \subseteq \ldots$ so
 that (i), (ii) and (iii) hold.  First we choose $l \in \mathbb{N}_0$ such that
 $\eta-\nicefrac{1}{m(1)} \le \nicefrac{l}{m(1)} \le \eta$ and pick a
 finitely generated closed subgroup $H_1 \le G$ such that $H_1 G_1 / G_1$ has
 order $p^l$.  This guarantees, in particular,
 that~\eqref{equ:inf-many} holds for $j=1$.

 Now suppose that finitely generated closed subgroups of $G$,
 \[
 1 = H_0 \subseteq H_1 \subseteq \ldots \subseteq H_j, \qquad
 \text{for some $j \in \mathbb{N}$,}
 \]
 have been constructed such that (i) and (ii) already hold for indices
 $i,i' \le j$ and such that \eqref{equ:inf-many} holds for $i=j$.  We
 manufacture finitely generated closed subgroups $H_{j+1},\ldots, H_k$
 of~$G$, for a suitable $k \in \mathbb{N}$ with $j < k$, to obtain an
 extended ascending chain
 \[
 1 = H_0 \subseteq H_1 \subseteq \ldots \subseteq H_k
 \]
 such that (i) and (ii) hold for indices $i,i' \le k$ and such that
 \eqref{equ:inf-many} holds for~$i=k$.

 \smallskip

 \noindent
 \underline{Case 1}. Suppose that
 \[
 \tfrac{1}{m(j+1)} \log_p \lvert H_j G_{j+1} : G_{j+1} \rvert \leq \eta.
 \]
 Using Lemma~\ref{lem:xyz-inequalities}(i), we observe
 that
 \begin{equation*}
   \tfrac{1}{m(j+1)} \log_p \lvert H_j G_j  : G_{j+1} \rvert
    = \frac{\log_p \lvert H_j G_j : G_j \rvert + \log_p \lvert G_j :
     G_{j+1} \rvert}{\log_p \lvert G : G_j \rvert + \log_p \lvert G_j :
     G_{j+1} \rvert} \ge \eta- \nicefrac{1}{m(j+1)}.
 \end{equation*}
 We may now take $k = j+1$ and finish the proof as follows.  Writing
 $l' = \log_p \lvert H_jG_k : G_k \rvert$ and
 $l'' = \log_p \lvert H_jG_{k-1} : G_k \rvert$, we find
 $l \in \mathbb{N}_0$ such that
 \[
 l' \le l \le l'' \qquad \text{and} \qquad \eta - \nicefrac{1}{m(j+1)}
 \le \nicefrac{l}{m(j+1)} \le \eta.
 \]
 Since $G/G_k$ is a finite $p$-group, we further find a
 finitely generated closed subgroup $H_k \le G$ with
 $H_j \subseteq H_k$ satisfying
 \[
 H_jG_k \subseteq H_k G_k \subseteq H_j G_{k-1} \qquad
 \text{and} \qquad \log_p \lvert H_k G_k : G_k \rvert = l.
 \]
 
 \smallskip

 \noindent
 \underline{Case 2}. Suppose that
 \[
 \tfrac{1}{m(j+1)} \log_p \lvert H_j G_{j+1} : G_{j+1} \rvert > \eta.
 \]
 Since $H_j$ is finitely generated, our hypotheses give
 $\hdim_G^\mathcal{S}(H_j) < \eta$.  Choose $k \in \mathbb{N}$
 with $k> j+1$ minimal subject to the condition
 \[
 \tfrac{1}{m(k)} \log_p \lvert H_j G_k : G_k \rvert < \eta.
 \]
 Setting $H_{j+1}=H_{j+2}=\ldots =H_{k-1}=H_j$, we observe that (i)
 and (ii) certainly hold for indices $i,i' \leq k-1$.  Using
 Lemma~\ref{lem:xyz-inequalities}(ii), we observe that
 \begin{equation*}
   \tfrac{1}{m(k)} \log_p \lvert H_j G_{k-1}  : G_k \rvert =
   \frac{\log_p \lvert H_j G_{k-1} : G_{k-1} \rvert + \log_p \lvert
     G_{k-1} : G_k \rvert}{\log_p \lvert G : G_{k-1} \rvert + \log_p \lvert G_{k-1} :
     G_k \rvert} \ge \eta \ge \eta - \nicefrac{1}{m(k)}.
 \end{equation*}
 We may now conclude the proof exactly as in Case~1.
\end{proof}

\begin{lemma} \label{LemmaB} Let $G$ be a countably based profinite
  group, and let
  $\mathcal{S} \colon G_0 \supseteq G_1 \supseteq \ldots$ be a
  filtration series of~$G$.  Let $H \le G$ be a closed subgroup such that
  $\hdim_G^\mathcal{S}(H)$ is given by a proper limit, that is
  \[
  \hdim_G^\mathcal{S}(H) = \lim_{i\rightarrow \infty} \frac{\log_p
    \lvert H G_i : G_i \rvert}{\log_p \lvert G : G_i \rvert}.
  \]
  Let $\mathcal{S} \vert_H \colon H = H_0 \supseteq H_1 \supseteq \ldots$,
  where $H_i = H \cap G_i$ for $i \in \mathbb{N}_0$, denote the
  induced filtration series of~$H$.
  Then for every $B\le H$,
  \[
  \hdim_G^\mathcal{S}(B) = \hdim_G^\mathcal{S}(H) \cdot
  \hdim_H^{\mathcal{S} \vert_H}(B).
  \]
\end{lemma}

\begin{proof}
  We observe that for $i \in \mathbb{N}_0$,
  \[
  \frac{\log_p \lvert B G_i : G_i \rvert}{\log_p \lvert G : G_i
    \rvert} = \frac{\log_p \lvert B H_i : H_i \rvert}{\log_p \lvert H
    : H_i \rvert} \frac{\log_p \lvert H : H_i \rvert}{\log_p \lvert G
    : G_i \rvert}.
 \]
 The claim follows by taking lower limits.
\end{proof}


\begin{theorem} \label{thm:eta-xi-interval} Let $G$ be a countably
  based pro-$p$ group, and let
  $\mathcal{S} \colon G_0 \supseteq G_1 \supseteq \ldots$ be a
  filtration series of~$G$.  Let $H \le G$ be a closed subgroup such
  that $\xi = \hdim_G^\mathcal{S}(H) > 0$ is given by a proper limit,
  that is
  \[
  \xi = \lim_{i\rightarrow \infty}\, \frac{\log_p \lvert H G_i :
    G_i \rvert}{\log_p \lvert G : G_i \rvert}.
  \]
  Suppose further that $\eta \in [0,\xi)$ is such that every finitely
  generated closed subgroup $K \le H$ satisfies
  $\hdim_G^\mathcal{S}(K) \le \eta$.  Then
  \[
  (\eta, \xi]\subseteq \hspec^\mathcal{S}(G).
  \]
\end{theorem}

\begin{proof}
  Let $\vartheta \in \mathbb{R}$ with $\eta <\vartheta <\xi$.  Let
  $\mathcal{S} \vert_H \colon H_0 \supseteq H_1 \supseteq \ldots$,
  where $H_i = H \cap G_i$ for $i \in \mathbb{N}_0$, denote the
  induced filtration series of~$H$.  By Lemma~\ref{LemmaB}, every
  finitely generated closed subgroup $K$ of $H$ satisfies
  $\hdim_H^{\mathcal{S} \vert_H}(K) \le \eta / \xi < \vartheta / \xi$.
  As $H$ is not finitely generated,
  Proposition~\ref{pro:fin-gen-less-eta} shows: there is a closed
  subgroup $B \le H$ such that
  $\hdim_H^{\mathcal{S} \vert_H}(B) = \vartheta / \xi$.  Using
  Lemma~\ref{LemmaB}, we deduce that
  \[
  \hdim_G^\mathcal{S}(B) =\hdim_G^\mathcal{S}(H) \cdot
  \hdim_H^{\mathcal{S} \vert_H}(B) = \vartheta. \qedhere
  \]
\end{proof}

\begin{proof}[Proof of Theorem~\ref{thm:soluble-interval}]
  Suppose that the finitely generated soluble pro-$p$ group $G$ is not
  $p$-adic analytic.  Consider the
  derived series of $G$, consisting of $G^{(0)} = G$ and
  $G^{(j)} = [G^{(j-1)},G^{(j-1)}]$ for $j \in \mathbb{N}$.

  Observe that $\hdim^\mathcal{S}_G(G^{(0)}) = 1$ and $\hdim^\mathcal{S}_G(G^{(j)}) = 0$ for
  all sufficiently large~$j$.  Hence there is a maximal integer
  $k \ge 0$ such that $\hdim^\mathcal{S}_G(G^{(k)}) = 1$.  We want to apply
  Theorem~\ref{thm:eta-xi-interval} for the closed
    subgroup $G^{(k)}$.  Quite trivially,
  \[
  \frac{\log_p \lvert G^{(k)} G_i : G_i \rvert}{\log_p \lvert G : G_i
    \rvert} \le 1 \qquad \text{for all $i \in \mathbb{N}_0$},
  \]
  where $G_i$ denotes the $i$th term of the given filtration
  series~$\mathcal{S}$.  Hence $\hdim^\mathcal{S}_G(G^{(k)})$ is actually given by
  a proper limit, that is
  \[
  \hdim^\mathcal{S}_G (G^{(k)}) = \lim_{i\rightarrow \infty} \frac{\log_p
    \lvert G^{(k)} G_i : G_i \rvert}{\log_p \lvert G : G_i \rvert}.
  \]

  Put $\eta = \hdim^\mathcal{S}_G(G^{(k+1)}) < 1$, and let
    $K \le G^{(k)}$ be any finitely generated closed subgroup. Then
    Proposition~\ref{pro:p-adic-analytic-qu-0}, applied to
    $N = G^{(k+1)}$ and $H = KN$, yields
    \[
    \hdim_G^\mathcal{S}(K) \leq \hdim_G^\mathcal{S}(H) =
    \hdim_G^\mathcal{S}(N) = \eta.
    \]
    Hence all conditions in Theorem~\ref{thm:eta-xi-interval} are satisfied.
\end{proof}


\end{document}